\documentclass[11pt, oneside]{article}

\usepackage{amsmath, amssymb, amsthm}
\usepackage{geometry}
\usepackage{enumitem}

\usepackage{graphicx}

\usepackage{capt-of}
\usepackage{tikz}
\usetikzlibrary{arrows,positioning,fit,shapes,calc}
\usetikzlibrary{decorations.markings}
\usepackage{ifthen}

\newtheoremstyle{thmstyle}
  {1.2\baselineskip}         
  {.5\baselineskip}         
  {\itshape}    
  {0pt}         
  {\bfseries}   
  {}            
  {.5em}        
  {}            

\newtheoremstyle{claimstyle}
  {.8\baselineskip}         
  {0\baselineskip}         
  {\itshape}    
  {0pt}         
  {\bfseries}            
  {}            
  {.5em}        
  {}            

\theoremstyle{thmstyle}
\newtheorem{theorem}{Theorem}

\newtheorem{lemma}[theorem]{Lemma}
\newtheorem{corollary}[theorem]{Corollary}
\newtheorem{conjecture}[theorem]{Conjecture}
\newtheorem{observation}[theorem]{Observation}
\newtheorem*{question*}{Question}
\theoremstyle{claimstyle}
\newtheorem{claim}{\emph{Claim}}
\newtheorem*{claim*}{\emph{Claim}}

\newtheorem{case}{Case}

\renewenvironment{proof}{\setcounter{claim}{0}\setcounter{case}{0}{\noindent \textbf{\textit{Proof.}}}}
{\hfill $\Box$\vspace*{0.1in}}

\newenvironment{proofof}[1]{{\noindent \textbf{\textit{Proof of #1.}}}}
{\hfill $\Box$\vspace*{0.1in}}

\newcommand{\vs}{\vspace*{.1in}}

\long\def\symbolfootnote[#1]#2{\begingroup\def\thefootnote{\fnsymbol{footnote}}
\footnote[#1]{#2}\endgroup}

\def\sm{\setminus}
\def\K{\vec{K}}

  \def\vertexscale{1}
  \def\arrowscale{1.2}
  \def\arrowstyle{latex'} 

  \pgfarrowsdeclarecombine{arrowstyle symmetric}{arrowstyle symmetric}{\arrowstyle}{\arrowstyle}{}{}

  \newcommand{\drawarc}
  {
    \draw[
      decoration={
        markings, 
        mark=at position 0.65 with {\arrow[scale=\arrowscale]{arrowstyle symmetric}};,
      },
      postaction=decorate
    ]
  }

  \newcommand{\drawdigon}
  {
    \draw[
      decoration={
        markings, 
        mark=at position 0.35 with {\arrowreversed[scale=\arrowscale]{arrowstyle symmetric}};,
        mark=at position 0.65 with {\arrow[scale=\arrowscale]{arrowstyle symmetric}};,
      },
      postaction=decorate
    ]
  }

  \newcommand{\drawvertex}[1]{\path #1 coordinate (vtemp);\filldraw[scale=\vertexscale] (vtemp) circle (0.05)}

\begin{document}

\title{Immersing complete digraphs}
\author{
Matt DeVos \footnote{mdevos@sfu.ca. Supported in part by an NSERC Discovery Grant (Canada) and a Sloan Fellowship.}\\
Jessica McDonald\footnote{jessica\textunderscore mcdonald@sfu.ca. Supported by an NSERC Postdoctoral Fellowship (Canada)}\\
Bojan Mohar \footnote{mohar@sfu.ca. Supported in part by an NSERC Discovery Grant (Canada), by the Canada Research Chair program, and by the Research Grant P1--0297 of ARRS (Slovenia). On leave from: IMFM \& FMF, Department of Mathematics, University of Ljubljana, Ljubljana, Slovenia.}\\
Diego Scheide \footnote{dscheide@sfu.ca}\\
\medskip\\
Department of Mathematics\\
Simon Fraser University\\
Burnaby, B.C., Canada V5A 1S6
}
\date{}

\maketitle

\bigskip

\begin{abstract}
We consider the problem of immersing the complete digraph
on $t$ vertices in a simple digraph. For Eulerian digraphs, we show that such an immersion always exists whenever minimum degree is at least
$t(t-1)$, and for $t \le 4$ minimum degree at least $t-1$ suffices.  On the other hand, we show that there exist non-Eulerian digraphs with all vertices of arbitrarily high in- and outdegree which do not contain an immersion of the complete digraph on $3$ vertices. As a side result, we obtain a construction of digraphs with large outdegree in which all cycles have odd length, simplifying a former construction of such graphs by Thomassen.
\end{abstract}

\section{Introduction}

In this paper all digraphs are finite and may have loops and multiple edges. A digraph $D$ is \emph{simple} if $D$ has no loops and there is at most one edge from $x$ to $y$ for any $x,y \in V(D)$. Note that oppositely oriented edges between two vertices are allowed in simple digraphs, and such a pair of edges is called a \emph{digon}.  The \emph{complete digraph} of order $t$, denoted $\K_t$, is a digraph with $t$ vertices and a digon between each pair of vertices.  The rest of our terminology is fairly standard, as used in \cite{BJ} or \cite{BM}.

A directed path of length two from a vertex $x$ to a vertex $z$ in a digraph $D$, say $xyz$, can be \emph{split off} by deleting the directed edges $xy$ and $yz$ and adding the directed edge $xz$ (possibly in parallel to existing directed edges). A digraph $F$ is said to be \emph{immersed} in a digraph $D$ if (a digraph isomorphic to) $F$ can be obtained from a subgraph of $D$ by splitting off directed paths of length two (and deleting isolated vertices). Equivalently, $F$ is immersed in $D$ if there exists an injective map $\varphi: V(F) \to V(D)$ and a collection of edge-disjoint directed paths in $D$, one from $\varphi(u)$ to $\varphi(v)$ for every edge $uv$ in $F$. In this setting, the vertices $\{\varphi(v): v\in V(F)\}$ are called \emph{terminals}. When the collection of directed paths for an immersion is internally disjoint from its set of terminals, the immersion is said to be \emph{strong}; otherwise it may be called \emph{weak}. We omit these descriptors in this paper, as we are only interested in the weaker notion.

Our main interest is in finding conditions on a digraph which imply the existence of certain immersions.  Although this appears to be a very natural problem it seems to have received rather little attention.  One exception to this is a result of Mader \cite{Ma2}, who proved that every digraph $D$ of minimum outdegree $t$ immerses the digraph consisting of two vertices $x,y$ and $t$ copies of the edge from $x$ to $y$ (or, equivalently, $D$ has $t$ edge-disjoint directed paths from $u$ to $v$ for some $u,v \in V(D)$).  In contrast to this, no outdegree assumption implies the existence of $\K_2^2$ -- the digraph consisting of two vertices $x,y$ and two copies of each non-loop edge $xy$ and $yx$; see Figure \ref{fig:K2^2}. This was first proved by Mader \cite{Ma3}, who utilized a family of digraphs constructed earlier by Thomassen \cite{Th} which have high outdegree but no cycle of even length.  We give a new proof of this result (highlighted below) which is based on a simpler construction.  In fact, our graphs also give an easy example of high outdegree digraphs with no even cycle, thus simplifying the former construction by Thomassen.  Since this latter property is not our focus, this argument is relegated to an appendix.

\begin{figure}[htb]
\begin{center}
  \def\vertexscale{1}
  \def\arrowscale{1.2}
  \def\arrowstyle{latex'} 
\begin{tikzpicture}[scale=1.3]
  \input{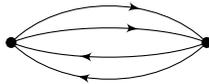}
\end{tikzpicture}
\caption{The digraph $\K_2^2$}
\label{fig:K2^2}
\end{center}
\end{figure}

\begin{theorem}
\label{thm:diego-digraph}
For every positive integer $k$ there exists a simple digraph $D$ with minimum in- and outdegree at least $k$ so that $D$ does not immerse $\K_2^2$.
\end{theorem}

Mader also considered the problem of finding subdivisions of transitive tournaments.  He proved in \cite{Ma4} that every digraph of minimum outdegree 3 contains a subdivision of the transitive tournament on 4 vertices, and he conjectured in \cite{Ma3} that there exists a function $f$ so that every digraph of minimum outdegree $f(t)$ contains a subdivision of the transitive tournament on $t$ vertices.  Weakening the conclusion of this conjecture to allow for an immersion of the transitive tournament on $t$ vertices  yields the following interesting open problem.

\begin{conjecture}
There exists a function $f$ so that every simple digraph of minimum outdegree $f(t)$ contains an immersion of the transitive tournament on $t$ vertices.
\end{conjecture}

In light of Theorem \ref{thm:diego-digraph} we cannot hope to find a subdivision or even an immersion of a complete digraph $\K_t$ for $t \ge 3$ without some assumption beyond a minimum degree condition.  One positive result of this type is a theorem of  K\"uhn, Osthus and Young \cite{KOY} which shows that every suitably dense digraph contains a subdivision of a large complete digraph.  Our main results show that a minimum degree condition together with the added assumption of Eulerian implies the existence of a $\K_t$ immersion. Given an Eulerian digraph, we will use the term \emph{degree} to refer to both the outdegree and the indegree of a vertex, and similarly for the terms \emph{minimum degree} and \emph{maximum degree}. 

\begin{theorem}
\label{thm:quadratic}
Every simple Eulerian digraph with minimum degree at least $t(t-1)$ contains an immersion of $\K_t$.
\end{theorem}

The quadratic bound of Theorem \ref{thm:quadratic} can be strengthened for small values of $t$ as follows.

\begin{theorem}
\label{thm:K_t for t<=4}
For $t \le 4$, every simple Eulerian digraph of minimum degree at least $t-1$ contains an immersion of $\K_t$.
\end{theorem}

In light of Theorem~\ref{thm:quadratic}, we may define a function $f$ by the rule that $f(t)$ is the smallest integer so that every simple Eulerian digraph of minimum degree at least $f(t)$ contains an immersion of $\K_t$.
Our results (together with the trivial lower bound) show that $t-1 \le f(t) \le t(t-1)$ and $f(t) = t-1$ for $t \le 4$.

For ordinary (undirected) graphs, there is an analogous function $g$ defined by the rule that $g(t)$ is the smallest integer so that every graph of minimum degree at least $g(t)$ contains an immersion of $K_t$
(note that we have dropped the assumption of Eulerian here).  Our understanding of this function is considerably better.  In particular, a recent theorem established by the current authors
together with Fox and Dvo\v{r}\'{a}k  \cite{DDFMMS} shows that $g(t) \le 200t$ so $g(t) = \Theta(t)$.  Further, a theorem of Lescure and Meyniel \cite{LM} (see also DeVos et al. \cite{DKMO}) shows that $g(t) = t-1$ for $t \le 7$, while an example due to Seymour shows that $g(t) \ge t$ for $t \ge 10$ (see \cite{DKMO} or \cite{LM}).

Although high in and outdegree at every vertex does not imply the existence of an immersion of $\K_3$ in a general digraph, it is possible to force the existence of such an immersion under a connectivity assumption. We say that a digraph $D$ is \emph{strongly $k$-edge-connected\/} if $D - S$ is strongly connected for every $S \subseteq E(D)$ with $|S| < k$. For a digraph $D$ and a vertex $v \in V(D)$ we define an \emph{arborescence with root} $v$ to be a subdigraph $T$ of $D$ which is a spanning tree in the underlying graph of $D$ and has the property that all edges are directed ``away'' from $v$. Note that $v$ is the unique vertex of $T$ with indegree $0$, and that all other vertices of $T$ have indegree $1$.  Edmonds' Disjoint Arborescence Theorem (\cite{Ed1}, \cite{Ed2}; see \cite{Sch} for a proof) gives a necessary and sufficient condition for the existence of a collection of edge-disjoint arborescences with prescribed root vertices in a general digraph.

For a digraph $D$ and a set $X \subseteq V(D)$ we let $d^+(X)$ denote the number of edges with initial point in $X$ and terminal point in $V \setminus X$ and we set $d^-(X) = d^+(V \setminus X)$.

\begin{theorem}[Edmonds' Disjoint Arborescence Theorem \cite{Ed1}, \cite{Ed2}]
\label{thm:arborescence}
Let $D$ be a digraph, and let $v_1, \ldots, v_{\ell} \in V(D)$ (not necessarily distinct).  Then there exist edge-disjoint arborescences $T_1, \ldots, T_{\ell}$ so that $T_i$ has root $v_i$ if and only if every
$X \subset V(D)$ satisfies the following condition:
\[ d^+(X)  \ge | \{ i :  v_i \in X, 1 \le i \le \ell \} |. \]
\end{theorem}

This result has the following easy corollary.

\begin{corollary}
\label{cor:edmonds-cor}
If $D$ is a strongly $t(t-1)$-edge-connected digraph with $|V(D)| \ge t$ then $D$ contains
an immersion of $\K_t$.
\end{corollary}

\begin{proof}
Choose distinct vertices $v_1, \ldots, v_t$ and apply Edmonds' Disjoint Arborescence Theorem to choose $t(t-1)$ edge-disjoint arborescences with exactly $t-1$ of them having root $v_i$ for
$1 \le i \le t$. For each $1 \le i \le t$ use the $t-1$ arborescences rooted at $v_i$ to find a set of
edge-disjoint paths from $v_i$ to each other vertex $v_j$.  All together, these paths give an immersion of $\K_t$.
\end{proof}

The following result shows that the connectivity condition of Corollary \ref{cor:edmonds-cor} is best possible up to a factor of 2.

\begin{theorem}
\label{thm:con-digraph}
For every positive integer $t \geq 3$ there exists a simple digraph which is strongly $\frac{1}{2}t(t-3)$-edge-connected and does not immerse $\K_t$.
\end{theorem}

Continuing in the vein of our earlier analysis, we define a function $h$ by the rule that
$h(t)$ is the smallest integer so that every simple digraph of strong edge-connectivity at least
$h(t)$ immerses $\K_t$. Corollary \ref{cor:edmonds-cor} together with Theorem \ref{thm:con-digraph} then shows that $\frac{1}{2}t(t-3) + 1 \le h(t) \le t(t-1)$.

Note that the upper and lower bounds on $h(t)$ differ only by a factor of approximately 2, while our bounds on $f(t)$ differ by a factor of $t$ when $t\geq 5$. It would be of significant interest to obtain better bounds on this function $f$.  In Section 4 of this paper we establish the exact values of $f(t)$ when $t\leq 4$ (Theorem \ref{thm:K_t for t<=4}), and offer an example indicating some of the difficulty involved in extending our methods to determine $f(5)$. The proof of the upper bound on $f$, Theorem \ref{thm:quadratic}, is the subject of Section 3. This proof uses  Corollary \ref{cor:edmonds-cor} as well as a structural result showing that every simple Eulerian digraph of minimum degree $t(t-1)$ immerses a strongly $t(t-1)$-edge-connected digraph on at least $t$ vertices. In the following section we construct examples to prove the lower bound on $h(t)$ (Theorem \ref{thm:con-digraph}), as well as the fact that arbitrarily high in- and outdegrees do not necessarily imply even a $\K_3$ immersion (Theorem \ref{thm:diego-digraph}). As we will see, the former examples are actually extensions of the latter.

\section{Examples}
\label{sec:ex}

We construct, for every positive integer $k$, a simple digraph $D_k$ as follows.  Take a rooted (undirected) tree of height $k$ where the root vertex has degree $k$, and every vertex at distance $j\geq 1$ from the root has degree $k-j+1$. Orient the edges of this tree so that it is an arborescence from the root (so all edges are directed ``away'' from the root).   See Figure~\ref{fig:tree structure for K3-free example} for an example with $k=4$.  Additionally, for every vertex $x$ (at level $j\geq 1$), add an edge from $x$ to every vertex on the path from the root to $x$ (a total of $j-1$ edges).

\begin{figure}[htb]
\begin{center}
  \def\vertexscale{0.8}
  \def\arrowscale{1.2}
  \def\arrowstyle{latex'} 
  \def\treeheight{4}
\begin{tikzpicture}[scale=1.3]
  \input{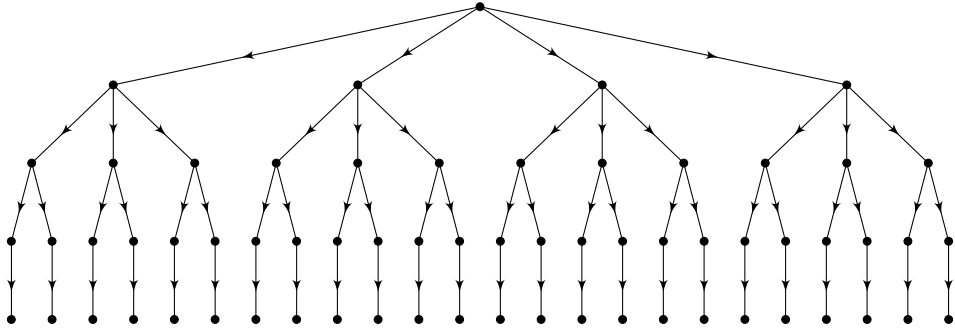}
  \begin{scope}
    \drawtree{\treeheight}{0}{0}{0.8}{10}
  \end{scope}
\end{tikzpicture}
\caption{Oriented rooted tree used to construct $D_\treeheight$}
\label{fig:tree structure for K3-free example}
\end{center}
\end{figure}

For two distinct vertices $x,y$ in a digraph $D$ we define $\lambda(x,y)$ to be the minimum of $d^+(X)$ over all subsets $X \subseteq V(D)$ with $x \in X$ and $y \not\in X$.  Menger's Theorem implies that $\lambda(x,y)$ is also equal to the maximum size of a collection of edge-disjoint directed paths from $x$ to $y$.

\begin{observation}
For every $k \ge 1$ the digraph $D_k$ has the following properties.
\begin{enumerate}
\item Every vertex in $D_k$ has outdegree $k$.
\item For every $u,v \in V(D_k)$ with $u \neq v$ we have $\min\{ \lambda(u,v), \lambda(v,u) \} \le 1$.
\item For every $v \in V(D_k)$ there exist $k$ edge-disjoint directed paths starting at $v$ with exactly one ending at each vertex at level 1 (i.e. the outneighbours of the root).
\end{enumerate}
\end{observation}

\begin{proof}
Property 1 follows immediately from the definition.  For property 2, note that there is an edge-cut separating $u$ and $v$ which intersects the rooted tree in a single edge $e$.  Since every edge other than $e$ has the opposite orientation to $e$ in this cut, we must have $\min\{ \lambda(u,v), \lambda(v,u) \} \le 1$.  For property 3, note that every vertex other than the root has an edge to the root.  So, for a non-root vertex $v$, there are exactly $k$ directed paths of length $\le 2$ from $v$ to the root and these are all edge-disjoint.  Extending each of these along a different edge away from the root gives a collection of $k$ edge-disjoint walks, which implies the existence of the desired paths.
\end{proof}

\begin{proofof}{Theorem~\ref{thm:diego-digraph}}
Let $D'_k$ be a copy of $D_k$ with all edges reversed, and construct the digraph $D$ from the disjoint union of $D_k$ and $D_k'$ by adding all possible edges from $D'_k$ to $D_k$.  It follows from the first part of the observation that all vertices of $D$ have indegree and outdegree at least $k$.  If $D$ contains a $\K_2^2$-immersion, then such an immersion must also exist in $D_k$, but this would contradict the second part of the observation.
\end{proofof}

\begin{proofof}{Theorem~\ref{thm:con-digraph}}
Since the Theorem is trivially true for $t=3$, we may assume $t\geq 4$. We begin with the digraph $D$ constructed in the proof of the previous theorem for the parameter
$k = \frac{1}{2}t(t-3)$.  Let $X$ denote the outneighbours of the root vertex in $D_k$ and let $X'$ denote the inneighbours of the root vertex in $D'_k$. Form the digraph $F$ from $D$ by adding a perfect matching $M$ between $X$ and $X'$ so that all edges are oriented from $X$ to $X'$.
The proof of the theorem will follow from two separate claims.
\begin{claim*} $F$ is strongly $k$-edge-connected.
\end{claim*}
To check this, we let $u,v \in V(F)$, and we shall show that there exist $k$ edge-disjoint walks from $u$ to $v$.  If $u,v \in D_k$ then we may use part 3 of the observation to choose $k$ edge-disjoint paths in $D_k$ starting at $u$ and ending at $X$.  Now, extend each of these along an edge of $M$ to obtain edge-disjoint walks ending at $X'$ and then extend each of these walks with endpoint $x'$ along the edge $x'v$ to obtain $k$ edge-disjoint walks form $u$ to $v$.  A similar argument works for $u,v \in D'_k$.  If $u \in D_k$ and $v \in D'_k$ then (again using part 3 of the observation) we may choose $k$ edge-disjoint paths from $u$ to $X$ in $D_k$ and $k$ edge-disjoint paths from $X'$ to $v$ in $D'_k$.  By adding the edges in our matching $M$ to these, we obtain $k$ edge-disjoint paths from $u$ to $v$ as desired.  Finally, if $u \in D_k'$ and $v \in D_k$ then there exist $k$ edge-disjoint single edge paths from $u$ to $X$, these may be extended using the edges of $M$ to $k$ edge-disjoint paths from $u$ to $X'$ and now since every vertex in $X'$ has an edge to $v$, they can be extended to give the desired paths.

\begin{claim*} $F$ does not contain an immersion of $\K_t$.
\end{claim*}
Suppose (for a contradiction) that $F$ contains an immersion of $\K_t$.  Since each edge of $M$ can only help in forming a single edge in this immersed digraph, it follows that $F - M$ must immerse a digraph $K'$ obtained from $\K_t$ by deleting $k$ edges.  Since $K'$ still has at least
$\binom{t}{2} - k = \frac{1}{2}t(t-1) - \frac{1}{2}t(t-3) = t$ digons, it follows that $K'$ must contain a subdigraph which is isomorphic to the digraph obtained from a cycle (of length $\ge 3$) by replacing each edge by a digon.  If $u,v$ are vertices in this subdigraph, then we must have $\lambda_{F-M}(u,v) \ge 2$ and $\lambda_{F-M}(v,u) \ge 2$.  However, this contradicts part 2 of the observation.
\end{proofof}

\section{A quadratic bound}

The goal of this section is to prove Theorem~\ref{thm:quadratic}.  Our proof will use the corollary of Edmonds' Disjoint Arborescence Theorem stated in the introduction, and also the following classical result of Mader on splitting.

\begin{theorem}[Mader's Directed Splitting Theorem \cite{Ma}]
Let $D$ be a digraph with a distinguished vertex $v$.  Suppose that ${\mathit deg}^+(v) = {\mathit deg}^-(v)$ and the following is satisfied:
\[ (\star) \;  \mbox{Every nonempty $X \subset (V(D) \setminus \{v\})$ satisfies $d^+(X) \ge k$ and $d^-(X) \ge k$.} \]
Then for every edge $uv$ there exists an edge $vw$ so that the new digraph formed by splitting off the path $uvw$ still satisfies property $(\star)$.
\end{theorem}

With this, we are ready to prove our main structural result from this section.   Extending our earlier notation, for a digraph $D$ and a set $X \subseteq V(D)$ we let $d(X) = d^+(X) + d^-(X)$.  Note that every Eulerian digraph satisfies $d^+(X) = d^-(X)$, so, in particular, $d(X)$ is always even.

\begin{theorem}
\label{thm:edge-connected immersion}
Every simple Eulerian digraph with minimum degree $r$ immerses a strongly $r$-edge-connected Eulerian digraph $F$ with $|V(F)| > r$ and the property that $F - S$ is simple for a set $S \subseteq E(F)$ with $|S| < r$
\end{theorem}

\begin{proof}
Let $D$ be a simple digraph with minimum degree $r$.  Set $X_{-1} = V(D)$ and choose a sequence of subsets $X_0, \ldots, X_{r-1}$
according to the rule that, for $0\leq k \leq r-1$:
\[ \mbox{$X_k$ is a minimal subset of $X_{k-1}$ with the property that $d(X_k) \le 2k$}. \]
Note that $d(X_{-1}) = 0$ and more generally $d(X_{k-1}) \leq 2k$ holds for all $k\geq0$ by induction, so the above choice is always possible. Note further that every $X_k$ with $k \ge 0$ induces a connected subdigraph.

\begin{claim*}
If $x,y \in X_k$ and $x \neq y$ then $\lambda_D(x,y) \ge k+1$.
\end{claim*}
To see this, suppose (for a contradiction) that it fails, and choose a set $Z \subseteq V(D)$ so that $|Z \cap \{x,y\}| = 1$ and $d^+(Z) \le k$ (note this last condition is equivalent to $d(Z) \le 2k$ since $D$ is Eulerian)
and subject to this the smallest index $j$ so that $Z \not\subseteq X_j$ is as large as possible.  We cannot have $Z \subseteq X_k$ as otherwise $Z$
would already contradict the choice of $X_k$, and thus $j$ exists.  Now, setting $\bar{Z} = V(D) \setminus Z$ and using the submodularity of the function $d$ gives us
\[ d(X_j \setminus Z) + d(Z \setminus X_j) = d(X_j \cap \bar{Z}) + d(X_j \cup \bar{Z}) \le d(X_j) + d(\bar{Z}) \leq 2j + 2k. \]
The set $Z \setminus X_j$ is a proper nonempty subset of $X_{j-1}$ so by assumption we have $d(Z \setminus X_j) \ge 2j$.  But then the above inequality
implies that the set $Z' = X_j \setminus Z$ satisfies $d(Z') \le 2k$ and then $Z'$ contradicts the choice of $Z$, thus completing the proof of the claim.

If $X_{r-1} = X_0$ then $D$ has a strongly $r$-edge-connected component, and we are already finished.  Otherwise, choose the smallest integer $j$
so that $X_{r-1} = X_j$ and note that $j \ge 1$ and $d(X_j) = 2j$ (were $d(X_j) < 2j$ we would have $X_j = X_{j-1}$).  Now, choose a vertex $x \in X_j$ and a vertex $y \in X_{j-1} \setminus X_j$.  It follows from the claim that $\lambda_D(x,y) \ge j$ so we may choose a collection $P_1, \ldots, P_j$ of edge-disjoint directed paths from $x$ to $y$.  It follows from the assumption that $D$ is Eulerian that the digraph obtained from $D$ by removing the edges of these paths has a collection $Q_1,\ldots,Q_j$ of edge-disjoint directed paths from $y$ to $x$.  Together, the existence of $P_1,\ldots,P_j$ and $Q_1,\ldots,Q_j$ implies that $D$ immerses the digraph $D'$ obtained from $D$ by identifying $V \setminus X_j$ to a single new vertex $w$ and deleting any resulting loops.
Since this identification can only increase the function $\lambda$ and since $X_j = X_{r-1}$ we then have $\lambda_{D'}(u,v) \ge \lambda_D(u,v) \ge r$ for every $u,v \in X_{r-1}$.  Now, we repeatedly apply Mader's Theorem to do splits at the vertex $w$ of $D'$ until $w$ becomes an isolated vertex, and we let $F$ be the component of the resulting digraph with vertex set $X_{r-1}$.  We now have $\lambda_{F}(u,v) \ge r$ for every $u,v \in X_{r-1}$ with $u \neq v$ so $F$ is strongly $r$-edge-connected.  Furthermore, setting $S$ to be the set of edges in $F$ formed by doing splits at $w$, we find that $|S| < r$ and $F - S$ is simple.  To complete the proof, set $n = |V(F)| \ge 2$ and note that $F - S$ has at most $n(n-1)$ edges which gives us
$n(n-1) + r >  |E(F)| = \sum_{v \in V(F)} {\mathit deg}^+_F(v) \ge rn$. This implies that $n > r$ as desired.

\end{proof}

\begin{proofof}{Theorem~\ref{thm:quadratic}}
We may assume that $t \ge 2$ as otherwise the result is trivial.  By Theorem~\ref{thm:edge-connected immersion} we may choose a digraph $D'$ immersed in $D$ which is strongly $t(t-1)$-edge-connected and has at least $t(t-1) \ge t$ vertices.  It now follows from Corollary \ref{cor:edmonds-cor} that $D'$ contains an immersion of $\K_t$ which completes the proof.
\end{proofof}

\section{Immersing small complete digraphs}

The following lemma shows that if we want to immerse $\K_t$ in an Eulerian digraph $D$, then we may as well assume that $D$ is regular.

\begin{lemma}\label{lemma:wma regular}
Let $t$ be a positive integer, and let $D$ be a simple Eulerian digraph with minimum degree at least $t$. Then $D$ contains an immersion of a simple $t$-regular Eulerian digraph.
\end{lemma}

\begin{proof}
  If $\Delta(D)=t$, then $D$ itself is $t$-regular and we are done. So we may assume $\Delta=\Delta(D)>t$.
  Let $v\in V(D)$ be a vertex with $\deg(v)=\Delta$. Since an inneighbour $u$ of $v$ has degree at most $\Delta$, there are at most $\Delta-1$ edges from $u$ to outneighbours of $v$.
  Hence, either $v$ has an inneighbour $u$ and an outneighbour $w\neq u$ with no edge from $u$ to $w$, or all inneighbours of $v$ are also outneighbours of $v$ and all possible edges between the neighbours of $v$ are present. In the first case we split off $uvw$ and use induction on a component of the resulting graph. In the second case we are done because $D$ contains $\K_{t+1}$ as a subgraph.
\end{proof}

We are now ready to show that both $\K_3$ and $\K_4$ immerse as we would like.

\vs

\begin{proofof}{Theorem~\ref{thm:K_t for t<=4}}
  For $t \leq 2$ the result is trivial, since an Eulerian digraph is strongly connected, and we only look for a cycle in $D$. Thus we only have to deal with $t=3$ and $t=4$. By Lemma~\ref{lemma:wma regular} it suffices to prove that any $(t-1)$-regular simple Eulerian digraph contains an immersion of $\K_t$. In fact, we will prove the following slightly stronger statement. Note that here, by \emph{parallel class of edges}, we mean a set of edges between two vertices that are oriented in the same direction.

  \noindent
  \emph{Let $t\in\{3,4\}$, let $D$ be an Eulerian digraph, and let $v_0\in V(D)$. Assume that $D$ satisfies the following conditions:
  \begin{enumerate}[label=$\bullet$, topsep=0pt, itemsep=0pt]
  \item
    $|V(D)| \geq 2$.
  \item
    $\deg(v_0) \leq t-1$ and $\deg(v)=t-1$ for all $v\in V(D)\sm\{v_0\}$.
  \item
    There is at most one parallel class of edges, and it is incident with $v_0$.
  \end{enumerate}
  Then $D$ contains an immersion of $\K_t$.
  }

  \medskip

  Assume (for a contradiction) that $D$ is a counterexample with $|V(D)|+|E(D)|$ minimum. We call $v_0$ the \emph{exceptional vertex}. We will prove four properties of $D$ before deducing a contradiction.

  \begin{claim}\label{claim:forced edges}
    If $v_0$ has an inneighbour $u$ and an outneighbour $w$ different from $u$, then there is an edge from $u$ to $w$.
  \end{claim}
  If Claim~\ref{claim:forced edges} is false, then let $D'$ be the graph obtained from $D$ by splitting off $uv_0w$. Clearly the component of $D'$ containing $u$ and $w$ is a smaller counterexample, where the exceptional vertex is either $v_0$ (if $v_0$ is in that component) or else any vertex.

  \begin{claim}\label{claim:no parallel class}
    There is no parallel class.
  \end{claim}
  If there is a parallel class, it has to be incident with $v_0$. By symmetry, we may assume that $v_0$ has two (or three) parallel in-edges from a vertex $u$. Since $\deg(u) \leq t-1 \leq 3$, it follows from Claim~\ref{claim:forced edges} that $v_0$ has at most one outneighbour different from $u$. Given that there is only one parallel class, this means that $\deg(v_0)=2$ and hence $N^+(v_0) = \{u,w\}$ for some vertex $w \neq u$. Note that, by Claim~\ref{claim:forced edges}, there is an edge $uw$, thus $\deg(u)=3$ and we are in the case $t=4$. Now let $D'$ be obtained from $D$ by splitting off $u v_0 w$ and deleting $v_0$ afterwards. Clearly, any immersion in $D'$ extends to an immersion in $D$. In $D'$ all vertices except $u$ have degree three, and there is one parallel class between $u$ and $w$. Thus $D'$ (with new exceptional vertex $u$) is a smaller counterexample.

  \begin{claim}\label{claim:deg(v_0)>=2}
    $\deg(v_0) \geq 2$.
  \end{claim}
  Since $|V(D)|\geq 2$ and $D$ is Eulerian we know that $\deg(v_0)>0$. If $\deg(v_0)=1$ with inneighbour $u$ and outneighbour $w$, then either $u=w$ (in which case $D-\{v_0\}$ with new exceptional vertex $u$ is a smaller counterexample), or $u \neq w$ (in which case splitting $v_0$ results in a smaller counterexample with new exceptional vertex $u$ or $w$). Note that in the former case, we maintain the required property of having at least two vertices since $u=w$ has degree $t-1 \geq 2$ in $D$.

  \begin{claim}\label{claim:deg(v_0)<t-1}
    $\deg(v_0) < t-1$.
  \end{claim}
  If $\deg(v_0)=t-1$ then, by Claim~\ref{claim:no parallel class}, $v_0$ has $t-1$ distinct outneighbours. The inneighbours of $v_0$ have to be identical to the outneighbours, otherwise there would be a vertex of degree at least $t$ by Claim~\ref{claim:forced edges}. Hence, again by Claim~\ref{claim:forced edges}, $v_0$ and its $t-1$ neighbours form a $\K_t$, a contradiction.

  \medskip

  By Claims~\ref{claim:deg(v_0)>=2} and \ref{claim:deg(v_0)<t-1} imply that $t=4$ and $\deg(v_0)=2$. By Claim~\ref{claim:no parallel class}, $v_0$ has two distinct inneighbours, say $v_1$ and $v_2$. We distinguish three cases, based on where the two outneighbours of $v_0$ are.

  \begin{case}
    $N^+(v_0)=\{v_1,v_2\}$.
  \end{case}
  By Claim~\ref{claim:forced edges}, $v_1$ and $v_2$ are bidirectionally joined. For $i=1,2$, let $u_i$ be the remaining inneighbour of $v_i$, and let $w_i$ be the remaining outneighbour of $v_i$. Now let $D'$ be the digraph obtained from $D$ by contracting the vertices $v_0,v_1,v_2$ to a single vertex $v'_0$. It is easy to see that any immersion in $D'$ extends to an immersion in $D$. In $D'$ all vertices except $v'_0$ have degree three. If $u_1 \neq u_2$ or $w_1 \neq w_2$ then $D'$ contains at most one parallel class, thus $D'$ (with exceptional vertex $v'_0$) is a smaller counterexample. Hence we may assume that $u_1=u_2$ and $w_1=w_2$. If $u_1=w_1$ then $D'-\{v'_0\}$ is a smaller counterexample (with new exceptional vertex $u_1=w_1$). Otherwise let $D''$ be the digraph obtained from $D'$ by splitting off all edges incident to $v'_0$. Note that $D''$ is immersed inn $D$, so a $\K_4$-immersion in $D''$ does not exist. Since $D''$ is $3$-regular and contains exactly one parallel class from $u_1$ to $w_1$, it is a smaller counterexample.

  \begin{case}
    $|N^+(v_0)\cap\{v_1,v_2\}|=1$, say $v_2\in N^+(v_0)$.
  \end{case}
  Let $v_3$ be the remaining outneighbour of $v_0$. By Claim~\ref{claim:forced edges}, we have edges from $v_1$ to $v_2$ and to $v_3$, and from $v_2$ to $v_3$. Hence, $v_2$ has one additional inneighbour $u\notin\{v_0,v_1\}$ and one additional outneighbour $w\notin\{v_0,v_3\}$. First consider the case that $u=v_3$ and $w=v_1$. Then the coboundary of $V_0=\{v_0,v_1,v_2,v_3\}$ consists of exactly two in-edges at $v_1$ and two out-edges at $v_3$. Now let $D'$ be the graph obtained from $D$ by contracting $V_0$ to a single vertex $v'_0$. Since $D[V_0]$ contains two edge-disjoint paths from $v_1$ to $v_3$, any immersion in $D'$ extends to an immersion in $D$. All vertices except $v'_0$ have degree three in $D'$, and there is no parallel class in $D'$. Thus $D'$ is a smaller counterexample. Hence we have $u\neq v_3$ or $w\neq v_1$. By symmetry, we may assume $u\neq v_3$. Now let $D'$ be the graph obtained from $D$ by splitting off $uv_2v_3$, $v_0v_2w$, and $v_1v_2v_0$, and removing $v_2$. Since $u\notin\{v_0,v_1,v_3\}$ and $w\neq v_3$, this split creates no loops and gives only one parallel class between $v_1$ and $v_0$. Since all degrees (except for $v_2$) stay the same, $D'$ (with exceptional vertex $v_0$) is a smaller counterexample.

  \begin{case}
    $N^+(v_0)\cap N^-(v_0)=\emptyset$.
  \end{case}
  Let $N^+(v_0)=\{v_3,v_4\}$. By Claim~\ref{claim:forced edges}, we have all edges from $v_1$ and $v_2$ to $v_3$ and $v_4$. Let $u_1,u_2,u_3$ be the three inneighbours of $v_1$. We may assume that one of them, say $u_3$, is not $v_4$, and one of the other two, say $u_2$ is not $v_3$. Also note that $v_2\notin\{u_1,u_2,u_3\}$, because we know all three outneighbours of $v_2$, and none of them is $v_1$. Now let $D'$ be the graph obtained from $D$ by splitting off $u_1v_1v_0$, $u_2v_1v_3$, and $u_3v_1v_4$. Since $u_1\notin\{v_0,v_2\}$, $u_2\notin\{v_0,v_2,v_3\}$ and $u_3\notin\{v_0,v_2,v_4\}$, no loops or parallel edges are created. Hence, $D'$ is a smaller counterexample.
\end{proofof}

One would hope that proof of Theorem~\ref{thm:K_t for t<=4} could be generalized to larger $\K_t$, in particular to $\K_5$. Consider the following observation about this immersion problem:  If every simple Eulerian digraph $D$ with minimum degree $d$ immerses $\K_t$, then the same conclusion must still hold under the weaker assumption
that $D$ has all but one vertex of degree at least $d$.  To see this, note that were $D$ to be a counterexample to this stronger property with the single vertex $v$ of degree less than $d$,
then taking $d$ disjoint copies of $D$ and identifying all copies of the vertex $v$ yields a digraph of minimum degree $d$ which still does not immerse $\K_t$.  In light of this, it is natural
to permit one exceptional vertex of low degree. However, the digraph below indicates some of the difficulty in using the approach of Theorem~\ref{thm:K_t for t<=4} for $\K_5$.

\begin{figure}[htb]
\begin{center}
  \def\vertexscale{1}
  \def\arrowscale{1.2}
  \def\arrowstyle{latex'} 
\begin{tikzpicture}[scale=1.3]
  \input{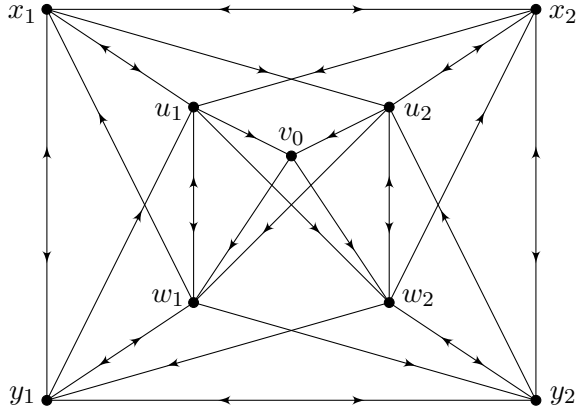}
\end{tikzpicture}
\caption{A difficult configuration for trying to immerse $\K_5$}
\label{fig:k5problems}
\end{center}
\end{figure}

In the digraph of Figure \ref{fig:k5problems}, all vertices have degree at least $4$ except for the vertex $v_0$, but however we choose a vertex and split off all of its edges, we move to a new digraph which has
at least two nontrivial disjoint parallel classes. Hence the structured assumption used to prove Theorem~\ref{thm:K_t for t<=4} does not extend easily for immersing larger complete digraphs.

\section*{Appendix}

Here we give a simplified proof of the following result:

\begin{theorem}\label{thm:Thom} \emph{(Thomassen \cite{Th})} For every $k$, there exists a digraph with minimum outdegree at least $k$ in which each cycle is of odd length.
\end{theorem}

Our graphs used to proved Theorem \ref{thm:Thom} are quite closely related to those constructed in Section \ref{sec:ex}, and the following  lemma contains the key property we require.

\begin{lemma}
Let $D$ be a digraph, let $F$ be an arborescence of $D$ and assume that for every edge $xy \in E(D) \setminus E(F)$ there exists a directed path in $F$ from $y$ to $x$.  Then every directed cycle in $D$ contains exactly one edge in $E(D) \setminus E(F)$.
\end{lemma}

\begin{proof}
Let $C$ be a directed cycle in $D$ and choose an edge $xy \in E(C) \setminus E(F)$.  Set $P$ to be the directed path from $y$ to $x$ in $F$.  Now, for every edge $e$ in $P$, the fundamental cut of $F$ with respect to this edge separates $x$ and $y$ and has the property that all edges other than $e$ in this cut are in the opposite direction to $e$. It follows from this that $e$ must appear in the cycle $C$.  But then $C$ must consist of precisely the edges in $P$ together with $xy$.
\end{proof}

Based on this simple lemma, we can easily construct digraphs of large outdegree without directed cycles of even length.  Let $F$ be an arborescence with the property that all leaf vertices are distance $2k$ from the root and all non-leaf vertices have outdegree exactly $k$.  Now, we add edges to $F$ to form a new digraph $D$ by the following rule.  For every directed path in $F$ of even length from a vertex $y$ to a leaf vertex $x$, we add the edge $xy$ to $D$.  It is immediate that every vertex in $D$ has outdegree $k$, and it follows from the above lemma that it has no directed cycle of even length.

\end{document}